\documentclass{article}
\usepackage{graphicx}
\usepackage{setspace}
\usepackage{amsmath}
\usepackage{amssymb,amsfonts,amsbsy}
\usepackage{amsthm}
\usepackage{mathtools}
\usepackage{amscd}
\usepackage{caption}
\usepackage[pdftex]{hyperref}
\usepackage{url}
\usepackage{float}
\usepackage{color}

\newcommand{\R}{\mathbb{R}}
\newcommand{\N}{\mathbb{N}}
\newcommand{\Z}{\mathbb{Z}}
\newcommand{\C}{\mathbb{C}}

\newcommand{\E}{\mathbb{E}}
\newcommand{\Eij}{E_{i_1j_1}\cdots E_{i_kj_k}}
\newcommand{\Tr}{\text{Tr}}

\newcommand{\Id}{\mathrm{Id}}
\newcommand{\la}{\lambda}

\newcommand{\U}{U(\mathfrak{gl}_N)}
\newcommand{\const}{\text{const}}
\newcommand{\ep}{\epsilon}
\newcommand{\ga}{\tau}

\newcommand{\Hom}{\mathrm{Hom}}
\newcommand{\LG}{L^2(G)}
\newcommand{\mW}{\mathcal{W}}

\newcommand{\wt}{\mathrm{wt}}

\newcommand{\bigO}[1]{\mathcal{O}\left({#1}\right)}
\newcommand{\state}[1]{\left\langle {#1} \right\rangle}
\newtheorem{theorem}{Theorem}[section]
\newtheorem{proposition}[theorem]{Proposition}

\newcounter{examplecounter}
\newenvironment{example}{{\noindent\ignorespaces}%
{\par\noindent%
\ignorespacesafterend}
    \refstepcounter{examplecounter}%
  \textbf{Example \arabic{examplecounter}}%
  \quad
}{

}
\newenvironment{remark}{{\noindent\ignorespaces}%
{\par\noindent%
\ignorespacesafterend}
  \textbf{Remark.}%
}{

}

\begin{document}
\title{Three--dimensional Gaussian fluctuations of non--commutative random surfaces along time--like paths}
\author{Jeffrey Kuan}
\maketitle

\abstract{We construct a continuous--time non--commutative random walk on $U(\mathfrak{gl}_N)$ with dilation maps $U(\mathfrak{gl}_N)\rightarrow L^2(U(N))^{\otimes\infty}$. This is an analog of a continuous--time non--commutative random walk on the group von Neumann algebra $vN(U(N))$ constructed in \cite{K}, and is a variant of discrete--time non--commutative random walks on $U(\mathfrak{gl}_N)$ \cite{B,CD}.

It is also shown that when restricting to the Gelfand--Tsetlin subalgebra of $U(\mathfrak{gl}_N),$ the non--commutative random walk matches a (2+1)--dimensional random surface model introduced in \cite{BF}. As an application, it is then proved that the moments converge to an explicit Gaussian field along time--like paths. Combining with \cite{BF} which showed convergence to the Gaussian free field along space--like paths, this computes the entire three--dimensional Gaussian field. In particular, it matches a Gaussian field from eigenvalues of random matrices \cite{Bo}. 

\section{Introduction}
Let us review some results in the mathematical and physics literature in order to motivate the problem. 

The \textit{Anisotropic Kardar--Parisi--Zhang} (AKPZ) equation, which was introduced in \cite{W} and is a variant of the KPZ equation first considered in \cite{KPZ}, describes a
 universal class of random surface growth models. Letting $h(t)$ denote the height of the surface at time $t,$ the equation in two dimensions is
$$
\partial_t h = \nu_x \partial_x^2 h + \nu_y \partial y^2 h + \frac{1}{2}\lambda_x(\partial_x h)^2 + \frac{1}{2}\lambda_y(\partial_y h)^2 + \eta, 
$$
where $\eta$ is space--time white noise and $\lambda_x,\lambda_y$ have different signs. (When $\lambda_x$ and $\lambda_y$ have the same sign, the equation is just the usual KPZ equation in two dimensions). Using non--rigorous methods, it was predicted (e.g. \cite{KS}) that  the stationary distribution for the AKPZ dynamics would be the Gaussian free field (see \cite{S} for a mathematical approach to the Gaussian free field). The question about the full three--dimensional process across different time variables remained open.  

However, the equation is mathematically ill--defined, due to the non--linear term. One mathematical approach is to consider exactly solvable models in the AKPZ universality class.  There have been two models considered, an interacting particle system and the eigenvalue process of a random matrix. Both will be described now. 

The interacting particle system, studied in \cite{BF}, lives on the lattice $\Z\times\Z_+$. It was shown that along \textit{space--like paths}, the particle system is a determinantal point process. (See Theorem \ref{SpaceLikePath} for the definition of space--like paths). By computing the correlation kernel and taking asymptotics, it was shown that the fluctuations of the height function of the particle system indeed converge to the Gaussian free field.  But due to the space--like path restriction, the problem of computing the limiting three--dimensional field remained unsolved. 

The random matrix model looks at the eigenvalues of minors of a large random matrix whose entries are evolving as Ornstein--Uhlenbeck processes. By a combinatorial argument, 
\cite{Bo} was able to compute the limiting three--dimensional Gaussian field, which has the Gaussian free field as a stationary distribution. The asymptotics at the edge were also computed in \cite{So}. However, one drawback is that the eigenvalues are not Markovian, as shown in \cite{ANv}. 

With these two models in mind, it is natural to want to consider an exactly solvable model that ``combines'' both models, and which is both Markovian and allows for the limiting three--dimensional field to be computed. This paper will construct such a model. 

Let us outline the body of the paper.  First, the model will be constructed as a continuous--time non--commutative random walk, which is a non--commutative version of the usual random walk in classical probability. The ``state space'' 
is $U(\mathfrak{gl}_N)$, the universal enveloping algebra of the Lie algebra $\mathfrak{gl}_N$ of $N\times N$ matrices. The dilation maps are algebra homomorphisms 
$j_n:U(\mathfrak{gl}_N)\rightarrow (M^{\otimes \infty},\omega)$, where $M$ is a von Neumann sub--algebra of the $U(\mathfrak{gl}_N)$--module $L^2(U(N))$ and $\omega$ is a state on $M^{\otimes\infty}$. These $j_n$ are a non--commutative analog of the usual definition of a stochastic process as a family of maps  $X_n$ from a probability space $(\Omega,\mathcal{F},\mathbb{P})$ to a state space $S$. It is proved below (Theorem 
\ref{QuantumTheorem}) that there is a semigroup of non--commutative Markov operator $\{P_t\}_{t\geq 0}$ on $U(\mathfrak{gl}_N)$  which is consistent with $j_n$. 

This model is analogous to a previously constructed non--commutative random walk on the group von Neumann algebra $vN(U(N))$ with dilation maps $vN(U(N))\rightarrow vN(U(N))^{\otimes\infty}$ \cite{K}. Additionally, it preserves the states from \cite{BB}. All of these construction involve a continuous family of characters of the infinite--dimensional unitary group $U(\infty)$. There have also been previous non--commutative random walks using the basic representation of $U(N)$ as input \cite{B,CD}.

It also turns out that $P_t$ preserves  $Z:=Z(U(\mathfrak{gl}_N))$, the centre of $U(\mathfrak{gl}_N)$.
This means that  $P_t\Big|_Z$ is a Markov operator in the usual (classical) sense. This Markov operator has a natural
description: By using the Harish--Chandra isomorphism which identifies $Z$ with the ring of shifted
symmetric polynomials in $N$ variables, $P_t$ can be identified with the Markov operator $Q_t$ of an
interacting system of $N$ particles on $\Z$. This is shown in Proposition \ref{OneLevel} below. This
interacting system is known as the \textit{Charlier Process}, see \cite{KOR}. In fact, the projection of
the interacting particle system from \cite{BF} onto $\Z\times\{N\}$ is exactly $Q_t$. When restricting our non--commutative
 random walk to the Gelfand--Tsetlin subalgebra, which is the subalgebra of $U(\mathfrak{gl}_N)$ generated by the centres $Z(U(\mathfrak{gl}_k)), 1\leq k\leq N$,
it also matches the two--dimensional particle system along space--like paths; see Theorem \ref{TwoLevels} for the precise statement.  It is worth mentioning 
that the matching most likely does not hold along time--like paths. 

We then take asymptotics of certain elements of the Gelfand--Tsetlin subalgebra and prove convergence to jointly 
Gaussian random variables. These elements correspond to moments of the random surface. Here, there 
is no requirement that the paths be space--like, allowing for convergence to Gaussians along time--like paths as well. The explicit 
covariance formula is given in Theorem \ref{TimeLikePath}. 

At first glance, it appears to be slightly different from the covariance formula for eigenvalues of random matrices. However, the process 
here corresponds to Brownian motion (see e.g. \cite{Bi2,CD}). Indeed, after applying 
the usual rescaling from Brownian motion to Ornstein--Uhlenbeck, the covariance from \cite{Bo} is 
recovered.

\textbf{Acknowledgments}. The author would like to thank Alexei Borodin, 
Alexey Bufetov, Philippe Biane and Ivan Corwin for enlightening discussions. 

\section{Preliminaries}
Let us review some background about representation theory and non--commutative random walks. See \cite{B2} for 
an introduction to non--commutative random walks.

\subsection{Representation Theory}
The universal enveloping algebra $\U$ is the unital algebra over $\C$ generated by $\{E_{ij},1\leq i,j\leq N\}$ with relations
$E_{ij}E_{kl}-E_{kl}E_{ij}=\delta_{jk}E_{il}-\delta_{il}E_{kj}$. It carries a natural $*$--operation induced from complex conjugation on $\C$. The coproduct
$\Delta: \U\rightarrow\U\otimes\U$ is the algebra morphism sending $E_{ij}$ to $E_{ij}\otimes 1 + 1\otimes E_{ij}$. There is a natural one--to--one correspondence between finite--dimensional $\U$--modules, finite--dimensional Lie algebra representations of $\mathfrak{gl}_N$, and finite--dimensional representations of the Lie group $G:=U(N)$.

Let $L^2(G)$ be the Hilbert space of square--integrable
complex--valued functions on $G$. Recall that by the
Peter--Weyl theorem, this Hilbert space has an orthogonal basis
given by the matrix coefficients of all irreducible
representations of $G$,
i.e.
$$
\{g \mapsto \eta(\pi_{\la}(g)\xi)\},
$$
where $\pi_{\la}$ runs over all irreducible representations of
$G$, $\{\xi\}$ runs over a basis for $V_{\la}$ and $\{\eta\}$
runs overs a basis for $V_{\la}^*$. Denote this basis by
$\{f_{\xi\eta}\}.$ Then there is a non--degenerate pairing
$\langle \cdot,\cdot\rangle$ between $\U$ and $L^2(G)$ given by
$$
\langle X, f_{\xi\eta} \rangle = \eta(X\xi).
$$
This can be heuristically understood as $\langle X,f\rangle = f(X),$ since $f_{\xi\eta}(g)=\eta(g\xi)$. This pairing defines
an injection $\U\hookrightarrow L^2(G)^*$. Let us review the
algebra structure of $L^2(G)^*$.

There is a co--algebra structure on $L^2(G)$ given by the
co--product $\Delta:\LG\rightarrow \LG\otimes\LG \cong
L^2(G\times G)$ defined by $\Delta(f)(x,y)=f(xy)$. The
multiplication $\mu$ on $\LG^*$ is the composition
$$
\LG^*\otimes \LG^* \stackrel{\rho}{\longrightarrow} (\LG\otimes
\LG)^* \stackrel{\Delta^*}{\longrightarrow} \LG^*,
$$
where $\rho(\phi\otimes\psi)(f\otimes h)=\phi(f)\psi(h)$. Use Sweedler's notation to write
$$\Delta(f) = \sum_{(f)} f_{(1)}\otimes f_{(2)}.$$
Evaluating both sides at $(x,y)\in G\times G$ shows 
$$
f(xy)=\sum_{(f)} f_{(1)}(x) f_{(2)}(y) \text{ for all } x,y\in
G.$$
Then
$$
\mu(\phi\otimes\psi)(f)=\Delta^*\rho(\phi\otimes\psi)(f)=\rho(\phi\otimes\psi)(\Delta
f) = \sum_{(f)} \phi(f_{(1)})\psi(f_{(2)}).
$$
In particular, if $\phi_x\in \LG^*$ denotes evaluation at $x$,
i.e. $\phi_x(f)=f(x)$, then
$$
(\phi_x\phi_y)(f) = \sum_{(f)} \phi_x(f_{(1)})\phi_y(f_{(2)}) =
\sum_{(f)} f_{(1)}(x)f_{(2)}(y) = f(xy).
$$
So $\phi_x\phi_y=\phi_{xy}$. We also write $\phi_{X}(\cdot)$
for$\langle X,\cdot\rangle.$

With the pairing between $\U$ and $\LG$ above, define the action $\pi$ of $\U$ on $\LG$ by
$$
\pi(a): f \mapsto \langle \mathrm{id} \otimes a , \Delta f\rangle.
$$
The symbol $\pi$ will sometimes be repressed, in the sense that $af$ means $\pi(a)f$. Observe that $\pi$  preserves each summand in the Peter--Weyl decomposition
$\LG=\bigoplus_{\la} V_{\la}^{(1)}\oplus\cdots\oplus
V_{\la}^{(\dim(\la))}$. To see this, suppose we are given some matrix coefficient in an irreducible representation $V_{\lambda}$, that is, an $f\in \LG$ of the form
$$
f(g) = \langle gv,w\rangle \text{ for fixed } v,w\in V_{\lambda}.
$$
Then by the definition of the co--product
$$
\sum_{(f)} f_{(1)}(g_1)f_{(2)}(g_2) = \langle g_1g_2 v,w\rangle.
$$
Since $\langle X, f_{(2)}\rangle = f_{(2)}(X)$, we see that 
\begin{equation}\label{Compare}
(\pi(X)f)(g) = \langle g\cdot Xv, w\rangle.
\end{equation}
Thus, $\pi(X)$ is of the form $\langle gv', w\rangle$ for $v',w\in V_{\lambda}$, so the summand is preserved. Letting be the von Neumann algebra consisting of the 
elements of $\Hom_{\C}(\LG,\LG)$ which preserve each summand in
the Peter--Weyl decomposition, we have that $\pi$ sends $\U$ to $M$. From the definition of the co--product in $\U$, the $n$--th tensor power $\pi^{\otimes n}:\U\rightarrow M$ 
is defined by
$$
\pi^{\otimes n}(X) = \sum_{i=1}^n \Id^{\otimes i-1}\otimes \pi(X) \otimes \Id^{\otimes n-i}.
$$

In general, any Lie group $G$ acts on its Lie algebra $\mathfrak{g}$ via the adjoint action
$$
\mathrm{Ad}(g)x = gxg^{-1}, \quad g\in G,x\in\mathfrak{g}.
$$
This action extends naturally to $\U$. For a subgroup $K$ of $G$, let
$\U^K=\{x\in\U: \mathrm{Ad}(g)x=x \text{ for all } g\in K\}$. In particular,
$\U^G=Z(\U)$, the centre of $\U$.

Recall that the Harish--Chandra isomorphism identifies $Z(\U)$ with shifted symmetric polynomials.
Explicitly, each $X\in Z(\U)$ acts as some constant $p_X(\la)$ on the irreducible representation 
$V_{\la}$. It turns out that $p_X$ is symmetric in the shifted variables $\la_i-i$. 

\subsection{Non--commutative probability}
A non--commutative probability space $(\mathcal{A},\phi)$ is a unital $^*$--algebra 
$\mathcal{A}$ with identity $1$ and a state $\phi:\mathcal{A}\rightarrow\C$, that is, a linear map such that $\phi(a^*a)\geq 0$ and $\phi(1)=1$. Elements of $\mathcal{A}$ are 
called \textit{non--commutative random variables}. 
This generalises a classical probability space, by considering $\mathcal{A}=L^{\infty}(\Omega,\mathcal{F},\mathbb{P})$ with $\phi(X)=\E_{\mathbb{P}}X$. We also need a notion of convergence. For a 
large parameter $L$ and $a_1,\ldots,a_r\in\mathcal{A},\phi$ which depend on $L$, as well as a limiting space $(\mathbb{A},\Phi)$, we say that $(a_1,\ldots,a_r)$ converges to $\mathbf{(a_1,\ldots,a_r)}$ with respect to the state $\phi$ if
$$
\phi(a_{i_1}^{\epsilon_1}\cdots a_{i_k}^{\epsilon_k})\rightarrow \Phi(\mathbf{a_{i_1}^{\epsilon_1}\cdots a_{i_k}^{\epsilon_k}})
$$
for any $i_1,\ldots,i_k\in\{1,\ldots,r\},\epsilon_j\in\{1,*\}$ and $k\geq 1$.

There is also a non--commutative version of a Markov chain. If $X_n:(\Omega,\mathcal{F},\mathbb{P})\rightarrow E$ denotes the Markov process with transition operator $Q:L^{\infty}(E)\rightarrow L^{\infty}(E)$, then the Markov property is
$$
\E[Yf(X_{n+1})]=\E[Y Qf(X_n)]
$$
for $f\in L^{\infty}(E)$ and $Y$ a $\sigma(X_1,\ldots,X_n)$--measurable random variable. Letting $j_n:L^{\infty}(E)\rightarrow L^{\infty}(\Omega,\mathcal{F},\mathbb{P})$ be defined by $j_n(f)=f(X_n)$, we can write the Markov property as
$$
\E[j_{n+1}(f)Y]=\E[j_n(Qf)Y]
$$
for all $f\in L^{\infty}(E)$ and $Y$ in the subalgebra of $L^{\infty}(\Omega,\mathcal{F},\mathbb{P})$ generated by the images of $j_0,\dots,j_n$. 

Translating into the non--commutative setting, we define a \textit{non--commutative Markov operator} to be a semigroup of
completely positive unital linear maps $\{P_t:t\in T\}$ from a $^*$--algebra $U$ to itself (not necessarily an algebra morphism). The set $T$ indexing time can be either 
$\N$ or $\R_{\geq 0}$, that is to say, the Markov process can be either discrete or continuous time.
If for any times $t_0<t_1<\ldots\in T$ there exists algebra morphisms  $j_{t_n}$ from  $U$ to a non--commutative probability space $(W,\omega)$ such that
$$
\omega(j_{t_n}(f)w)=\omega(j_{t_{n-1}}(P_{t_n-t_{n-1}}f)w)
$$
for all $f\in U$ and $w$ in the subalgebra of $W$ generated by the images of $\{j_t:t\leq t_{n-1}\}$, then $j_t$ is called a \textit{dilation} of $P_t$.

\section{Non--commutative random walk on $\U$}
The first thing that needs to be done is to define states on $\U$. Note that is already has a natural $^*$--algebra structure, 

Given any positive type, normalized (sending the identity to
$1$), class function $\kappa\in \LG$, we have the decomposition
$$\kappa = \sum_{\la\in \widehat{G}}\widehat{\kappa}(\la)\frac{\chi_{\la}}{\dim\la},$$ where
$\widehat{G}$ denotes the set of equivalence classes of irreducible representations of $G$, 
and $\chi_{\la}$ are the characters corresponding to $\la$. By the orthogonality relations, 
$\widehat{\chi_{\la}}(\la)=1$. This
defines a state $\kappa$ on $M$ by
$$
\kappa(X) = \sum_{\la} \widehat{\kappa}(\la)
\sum_{i=1}^{\dim\la} \Tr(X\vert_{V_{\la}^{(i)}}), \quad X\in M.
$$
This naturally pulls back via $\pi:\U\rightarrow M$ to a state $\kappa(\cdot)$ on $\U$. 

Recall an equivalent definition of these states from \cite{BB}. 
There is a canonical isomorphism
$D:U(\mathfrak{gl}(N))\rightarrow \mathcal{D}(N)$ where
$\mathcal{D}(N)$ is the algebra of left--invariant differential
operators on $U(N)$ with complex coefficients. Then the
state $\langle \cdot \rangle_{\kappa}$ on $U(\mathfrak{gl}(N))$
is defined by $$
\langle X\rangle_{\kappa} = D(X)\kappa(U)\vert_{U=I}.
$$
The state can be computed using the formula (see e.g. page 101
of \cite{kn:V}) for $X=E_{i_1j_1}\cdots E_{i_kj_k}$:
\begin{equation}\label{UsingTheFormula}
 D(X)\kappa(U) =
\partial_{t_1}\cdots\partial_{t_k}\kappa\left(U e^{t_1E_{i_1j_1}}\cdots
e^{t_kE_{i_kj_k}}\right)\Big|_{t_1=\cdots=t_k=0}.
\end{equation}
Comparing \eqref{UsingTheFormula} and \eqref{Compare} shows that $D=\pi$.
Here, $e^{tE_{ij}}$ is just the usual exponential of matrices,
which has the simple expression
\begin{equation}\label{SimpleExpression}
e^{tE_{ij}} = 
\begin{cases}
Id + tE_{ij}, \ \ i\neq j\\
Id + (e^t -1)E_{ii} \ \ i=j
\end{cases}
\end{equation}
Note that since \eqref{UsingTheFormula} only involves linear
terms in the $t_j$, one can replace $e^{tE_{ii}}$ with $Id +
tE_{ii}$ without changing the value of the right hand side of
\eqref{UsingTheFormula}.
This is a slightly different approach from \cite{BB}, which used the (equivalent) formula
$$
E_{ij} \mapsto \sum_k x_{ik}\partial_{jk}.
$$

It is not hard to see that the two definitions of $\state{\cdot}_{\kappa}$ are equivalent. For each $\la\in \widehat{G}$ and $X=\Eij$, and 
letting $v_1,\ldots,v_d$ be a basis of $V_{\la}$,
\begin{eqnarray*}
\state{X}_{\chi_{\la}} &=& \partial_{t_1}\cdots\partial_{t_k}\chi_{\la}\left(e^{t_1E_{i_1j_1}}\cdots
e^{t_kE_{i_kj_k}}\right)\Big|_{t_1=\cdots=t_k=0}.\\
&=& \partial_{t_1}\cdots\partial_{t_k}\sum_{r=1}^d \Big\langle e^{t_1E_{i_1j_1}}\cdots
e^{t_kE_{i_kj_k}}v_r,v_r \Big\rangle\Big|_{t_1=\cdots=t_k=0}\\
&=& \sum_{r=1}^d \Big\langle \Eij v_r,v_r \Big\rangle\\
&=& \Tr\left(X\big|_{V_{\la}}\right)
\end{eqnarray*}
By linearity, this holds for all $\kappa$ and all $X$.

Now that the states have been defined, we define the non--commutative Markov process. 
In order to define a continuous--time non--commutative Markov process, there needs to be
a semigroup $\{\kappa_t:t\geq 0\}$ in $\LG$. Indeed, such a semigroup exists: for any $t\geq 0$, let 
$$
\kappa_t(U)=e^{t\Tr(U-\Id)}.
$$
Now fix times $t_1<t_2<\ldots$. Let $\mW$ be the infinite tensor product of von Neumann algebras $M^{\otimes\infty}$
with respect to the state $\omega=\kappa_{t_1}\otimes\kappa_{t_2-t_1}\otimes\kappa_{t_3-t_2}\otimes\ldots$. 
For $n\geq 1$
define the morphism $j_{t_n}:\U\rightarrow\mW$ to be the map
$j_{t_n}(X)=\pi^{\otimes n}(X)\otimes \mathrm{Id}^{\otimes\infty},$
and let $\mW_n$ be the subalgebra generated by the images of
$j_{t_1},\ldots,j_{t_n}$. Define $P_t:\U\rightarrow\U$ by
$(\mathrm{id}\otimes \kappa_t)\circ\Delta$. Note that $P$ is
linear as a map of complex vector spaces, but is not an algebra
morphism, since the trace is not preserved under multiplication
of matrices. 
To simplify notation, write $\state{\cdot}_t$ for $\state{\cdot}_{\kappa_t}$
and $j_n$ for $j_{t_n}$.

\begin{theorem}\label{QuantumTheorem}
(1) The maps $(j_n)$ are a dilation of the
non--commutative Markov operator $P_t$. In other words,
$$
\omega(j_n(X)w)=\omega(j_{n-1}(P_{t_n-t_{n-1}}X)w), \quad X\in
U(\mathfrak{gl}(N)), \quad w\in \mW_{n-1}.
$$

(2) The pullback of $\omega$ under $j_n$ is the state $\langle
\cdot \rangle$ on $U(\mathfrak{gl}(N))$, i.e. $\langle
X\rangle_{t_n} = \omega(j_n(X))$.

(3) For $n\leq m,$ we have
$$
\omega(j_n(X)j_m(Y))=\langle X\cdot P_{t_m-t_n}Y\rangle_{t_n}.
$$

(4) The non--commutative markov operators $P_t$ satisfy the
semi--group property $P_{t+s}=P_{t}\circ
P_{s}$.

(5) For any subgroup $K\subset U(N)$, the restriction of $P_t$ to $\U^K$ is still a
non--commutative transition kernel. In other words, $P_t \U^K\subset \U^K$. 
In particular, $P_t Z(\U)\subset Z(\U)$.
\end{theorem}
\begin{proof}
(1) This is essentially identical to Proposition 3.1 from \cite{CD}. It is reproduced here
for completeness.
The left--hand--side is 
\begin{multline*}
\omega((\pi^{\otimes n-1}\otimes\pi)\Delta X w) = \sum_{(X)} \omega(\pi^{\otimes n-1}(X_{(1)})\otimes \pi(X_{(2)}) w) \\
= \sum_{(X)} \omega(\pi^{\otimes n-1}(X_{(1)})w) \state{X_{(2)}}_{t_n-t_{n-1}}.
\end{multline*}
The right--hand--side is
$$
\sum_{(X)} \omega(j_{n-1}(\state{X_{(2)}}_{t_n-t_{n-1}}X_{(1)}w) = \sum_{(X)} \omega(j_{n-1}(X_{(1)})w)\state{X_{(2)}}_{t_n-t_{n-1}}.
$$

(2) Let $m_n$ denote the $n$--fold multiplication $\LG^{\otimes
n}\rightarrow\LG$ that sends $f_1\otimes\cdots \otimes f_n$ to
$f_1\cdots f_n$. We will show that
\begin{equation}\label{Eqn1}
m_n(\pi^{\otimes n}(X)(f_1\otimes\cdots\otimes
f_n))=\pi(X)(f_1\cdots f_n).
\end{equation}
The case of general $n$ follows inductively from $n=2$. The left hand side is
$$
\sum_{(X)} (\pi(X^{(1)})f_1) \cdot (\pi(X^{(1)})f_2) = \sum_{(X,f_1,f_2)} \langle X^{(1)}, f_1^{(2)} \rangle f_1^{(1)} \cdot \langle X^{(2)}, f_2^{(2)} \rangle f_2^{(1)}.
$$
The right hand side is
$$
\langle \mathrm{id} \otimes X, \Delta(f_1\cdot f_2) \rangle = \sum_{(f_1,f_2)} \langle X, f_1^{(2)}f_2^{(2)} \rangle f_1^{(1)} \cdot f_2^{(1)}.
$$
So it suffices to show that
$$
\sum_{(X)} \langle X^{(1)}, f_1^{(2)} \rangle \langle X^{(2)}, f_2^{(2)} \rangle = \langle X, f_1^{(2)}f_2^{(2)} \rangle.
$$
But this is just the definition of multiplication in a dual Hopf algebra. So \eqref{Eqn1} is true.

Now recall that if $A$ is a Hopf algebra with co--unit
$\epsilon:A\rightarrow\C$, then the $n$--fold tensor power
$A^{\otimes n}$ is also a Hopf algebra, with co--unit
$\epsilon^{(n)}:A^{\otimes n}\rightarrow \C$ defined by the
composition
$$
A^{\otimes n}\xrightarrow{m_n}A
\stackrel{\epsilon}{\longrightarrow}\C.
$$
In other words,
$$
\epsilon^{(n)}(a_1\otimes \cdots \otimes a_n)=\epsilon(a_1
\cdots a_n) = \epsilon(a_1)\cdots \ep(a_n).
$$
The second equality holds because $\ep$ is a morphism of
$\C$--algebras.

When $A=\LG,$ then $\ep:\LG\rightarrow\C$ is defined by
$\ep(f)=f(Id_G)$. I now claim that for $X\in \U$
\begin{equation}\label{Eqn2}
\omega(\pi^{\otimes n}(X)) =
\ep^{(n)}X(\kappa_{t_n}).
\end{equation}
For $n=1$, this follows immediately from the definitions. For
$n\geq 2$, write as usual
$$
\pi^{\otimes n}(X)=\sum_{(X)} X_{(1)}\otimes\cdots\otimes
X_{(n)} \in M^{\otimes n}.
$$
Then
\begin{eqnarray*}
\omega(\pi^{\otimes n}(X)) &=& \sum_{(X)}
\kappa_{t_1}(X_{(1)})\cdots\kappa_{t_n-t_{n-1}}(X_{(n)}) \\
&=& \sum_{(X)} \ep
X_{(1)}(\kappa_{t_1}) \cdots \ep X_{(n)}(\kappa_{t_n-t_{n-1}}).
\end{eqnarray*}
At the same time, 
\begin{eqnarray*}
\ep^{(n)}X(\kappa_{t_n}) &=& \ep^{(n)}X(\kappa_{t_1}\cdots \kappa_{t_n-t_{n-1}}) \\
&=& \ep^{(n)}\sum_{(X)} X_{(1)}(\kappa_{t_1}) \otimes \cdots \otimes
X_{(n)}(\kappa_{t_n-t_{n-1}})\\
&=& \sum_{(X)} \ep X_{(1)}(\kappa_{t_1}) \cdots \ep
X_{(n)}(\kappa_{t_n-t_{n-1}}).
\end{eqnarray*}
So \eqref{Eqn2} is true.

Finally, we can combine the results to obtain
\begin{align*}
\omega(j_n(X))=\omega(\pi^{\otimes n}(X)) &=
\ep^{(n)}\pi^{\otimes n}X(\kappa_{t_1}\otimes\cdots\otimes \kappa_{t_n-t_{n-1}}) \\
&= \ep m_n(\pi^{\otimes
n}X(\kappa_{t_1}\otimes\cdots\otimes \kappa_{t_n-t_{n-1}}))\\
&=\ep \pi(X)(\kappa_{t_1}\otimes\cdots\otimes \kappa_{t_n-t_{n-1}}) =\langle X\rangle_{\kappa_{t_n}}.
\end{align*}
This is exactly what (2) stated.

(3) By repeated applications of (1),
$$\omega(j_n(X)j_m(Y))=\omega(j_n(X)j_n(P_{t_m-t_n}Y)).$$ Since $j_n$
is a morphism of algebras, this equals $\omega(j_n(X\cdot
P_{t_m-t_n}Y))$, which by (2) equals the right--hand--side of (3).

(4) By linearity, it suffices to prove this for monomials of the form
$E=\Eij$. Introduce some notation: let $K=\{1,\cdots,k\}$ and
for any subset $S\subseteq K$, define
$
E_S = \prod_{s\in S} E_{i_sj_s},
$
where the product is taken in increasing order. The term
$E_{\emptyset}$ is understood to be $1$. So, for example, if
$E=E_{13}E_{42}E_{55}E_{12}$ and $S=\{1,2,4\}$ then
$E_S=E_{13}E_{42}E_{12}$. With this notation,
$$
\Delta E = \sum_{S\subseteq K} E_S\otimes E_{K\backslash S}.
$$
And therefore
\begin{align*}
P_{t_2-t_1}E&=\sum_{S\subseteq K} \langle E_{K\backslash
S}\rangle_{t_2-t_1} E_S,\\
\quad P_{t_1}P_{t_2-t_1}E &=
\sum_{\substack{ S\subseteq K \\ R\subseteq S}}
\state{E_{S\backslash R}}_{t_1} \state{E_{K\backslash
S}}_{t_2-t_1} E_R.
\end{align*}
Since 
$$
P_{t_2}E = \sum_{R\subseteq K}
\state{E_{K\backslash R}}_{t_2}E_R,
$$
it suffices to show 
$$
\state{E_{K\backslash R}}_{t_2} = \sum_{R\subseteq
S\subseteq K} \state{E_{K\backslash S}}_{t_2-t_1}
\state{E_{S\backslash R}}_{t_1} \quad \text{for all }
R\subseteq K,
$$
or equivalently
$$
\state{E_{K}}_{t_2} = \sum_{S\subseteq K}
\state{E_{K\backslash S}}_{t_2-t_1} \state{E_{S}}_{t_1}.
$$
This follows from \eqref{UsingTheFormula} and the general
Leibniz rule applied to the derivatives of the product
$\kappa_{t_1}\cdot\kappa_{t_2-t_1}$.


(5) This is Proposition 4.3 from \cite{CD}. Here is also a bare--bones proof when $K=G$
Let $X\in Z(\U).$ The goal is to show that $P(X)Y=YP(X)$
for all $Y\in \U$. It suffices to show this when
$Y\in\mathfrak{g}$. In this case,
\begin{align*}
&XY=YX \Longrightarrow \Delta(XY)=\Delta(YX) \Longrightarrow
\Delta(X)\Delta(Y)=\Delta(Y)\Delta(X)\\
&\Longrightarrow \sum_{(X)} X_{(1)}Y\otimes X_{(2)} +
X_{(1)}\otimes X_{(2)}Y = \sum_{(X)} YX_{(1)}\otimes X_{(2)} +
X_{(1)}\otimes YX_{(2)}.
\end{align*}
Now apply the linear map $id \otimes \state{\cdot}$ to both
sides to get
$$
\sum_{(X)} \state{X_{(2)}}X_{(1)}Y + \state{X_{(2)}Y}X_{(1)} =
\sum_{(X)} \state{X_{(2)}}YX_{(1)} + \state{YX_{(2)}}X_{(1)}.
$$
Since the state $\state{\cdot}$ is tracial, the second summand
on both sides are equal. The first summand on the
left--hand--side is $P(X)Y$ while the first summand on the
right--hand--side is $YP(X)$, so $P(X)\in Z(\U)$ as needed.
\end{proof}

\section{Connections to classical probability}
In this section, we will show that restricting to the centres $Z(U(\mathfrak{gl}_1),$ $\ldots,Z(U(\mathfrak{gl}_N))$ reduces the non--commutative random walk to a (2+1)--dimensional random surface growth model. First, here is a description of the model, which was introduced in \cite{BF}.
\subsection{Random surface growth}
Consider the two--dimensional lattice $\Z\times\Z_+$. On each horizontal level $\Z\times\{n\}$ there are
exactly $n$ particles, with at most one particle at each lattice site. Let $X^{(n)}_1>\ldots>X^{(n)}_n$ 
denote the $x$--coordinates of the locations of the $n$ particles. Additionally, the particles need to 
satisfy the \textit{interlacting property} $X^{(n+1)}_{i+1} < X^{(n)}_i \leq X^{(n+1)}_i.$  
The particles can be viewed as a random stepped surface, see Figure \ref{Figure}. This can be made rigorous by defining the height function at $(x,n)$ to be the number of particles to the right of $(x,n)$.

\begin{figure}[H]
\captionsetup{width=0.8\textwidth}
\centering

\caption{ The particles as a stepped surface.  The lattice is shifted to make the visualization easier.}

\includegraphics[height=5cm]{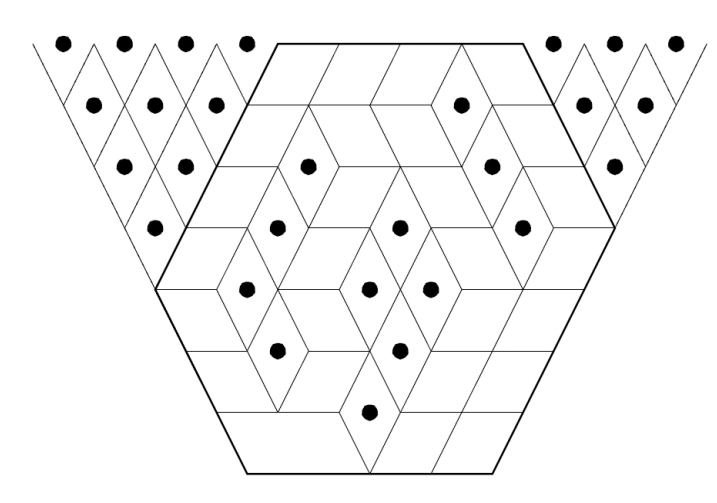}

\label{Figure}
\end{figure}

The dynamics on the particles are as follows. The initial condition is the \textit{densely packed} initial
condition, $\Lambda_i^{(n)}=-i+1,1\leq i\leq n$. Each particle has a clock with exponential waiting time of 
rate $1$, with all clocks independent of each other. When the clock rings, the particle attempts to jump
one step to the right. However, it must maintain the interlacing property. This is done by having  
particles push particles above it, and jumps are blocked by particles below it. One can think of lower
particles as being more massive. See Figure \ref{Jumps} for an example.

The projection to $\Z\times\{n\}$ is still Markovian, and is known as the \textit{Charlier process}
\cite{KOR}. It can be described by as a continuous--time Markov chain on $\Z^n$ with independent increments 
$e_i/n,1\leq i\leq n$, (where $\{e_i\}$ is the canonical basis for $\Z^n$) conditioned to stay in the Weyl chamber $(x_1>x_2>\ldots>x_n)$. Equivalently, the
conditioned Markov chain is the Doob $h$--transform for some harmonic function $h$. There is a nice
description of $h$ in terms of representation theory, namely, $h(x_1,\ldots,x_n)$ is the dimension of the
irreducible representation of $\mathfrak{gl}_n$ with highest weight $(x_1,x_2+1,\ldots,x_n+n-1)$. Explicitly, $$
\dim\la = \prod_{i<j} \frac{\lambda_i-i-(\lambda_j-j)}{j-i}.
$$
Below, let $Q_t^{(N)}$ denote the Markov operator of this Markov chain.

\begin{figure}
\captionsetup{width=0.8\textwidth}
\caption{The red particle makes a jump. If any of the black particles attempt to jump, their jump is 
blocked by the particle below and to the right, and nothing happens. White particles are not blocked.}
\centering
\includegraphics[height=6cm]{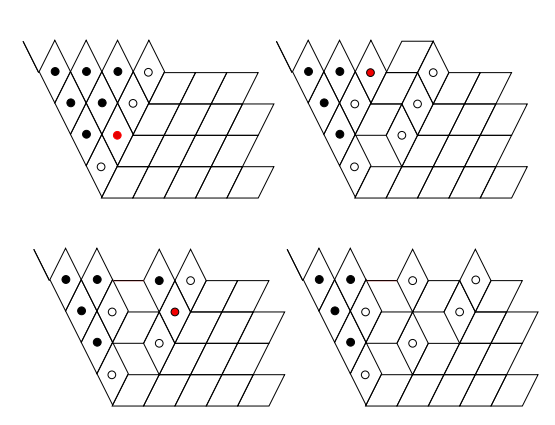}

\label{Jumps}
\end{figure}

The construction of the full particle system is based on a general multi--variate construction from
\cite{BF}, which is based on \cite{DF}. Suppose there are two Markov chains with state spaces
$\mathcal{S}, \mathcal{S}^*$ and transition probabilities $P,P^*$. Also assume there is a Markov operator
$\Lambda:\mathcal{S}^*\rightarrow\mathcal{S}$ which intertwines with $P,P^*$ in the sense that $\Lambda P^* = P\Lambda$.
In other words, there is a commutative diagram

\begin{equation}\label{CD}
\begin{CD}
\mathcal{S}^* & @>P^*>> & \mathcal{S}^*\\
 @VV\Lambda V &    & @VV\Lambda V\\
\mathcal{S} & @>P>> & \mathcal{S}
\end{CD}
\end{equation}

\vspace*{1\baselineskip}

Then the state space is $\{(x^*,x)\in \mathcal{S}^*\times\mathcal{S}: \Lambda(x^*,x)\neq 0\}$ with transition probabilities
$$
\mathrm{Prob}((x^*,x)\rightarrow (y^*,y)) = 
\begin{cases}
\frac{P(x,y)P^*(x^*,y^*)\Lambda(y^*,y)}{\Delta(x^*,y)},& \quad \Delta(x^*,y)\neq 0\\
0,& \quad \Delta(x^*,y)=0
\end{cases}
$$
Additionally, if the intial condition is a \textit{Gibbs measure}, that is, a probability distribution of the form $\mathbb{P}(x^*)\Lambda(x^*,x)$, then the dynamics preserves Gibbs measures. All constructions and definitions extend naturally to any finite number of Markov chains. 

Here, $Q_t^{(N)}$ and $Q_t^{(N-1)}$ will play the roles of $P^*,P$, and the projection $\Lambda$ is 
$$
\Lambda(x_1>\ldots>x_N,y_1>\ldots>y_{N-1}) = \frac{h(y)}{h(x)}.
$$
The construction implies that
\begin{equation}\label{Gibbs}
\mathbb{P}(X^{(N)}(t) = x^{(N)} \vert X^{(M)}(t) = x^{(M)}) = \frac{h(x^{(N)})}{h(x^{(M)})} , \forall N\leq M, t\geq 0
\end{equation}
\begin{multline}\label{Push}
\mathbb{P}(X^{(N)}(t) = x^{(N)} \vert X^{(M)}(s) = y^{(M)},X^{(N)}(s) = y^{(N)}) \\
= \mathbb{P}(X^{(N)}(t) = x^{(N)} \vert X^{(N)}(s) = y^{(N)}), \forall N\leq M,s\leq t
\end{multline}
The intuition behind \eqref{Push} is that since particles on lower levels push and block the particles on higher levels, the evolution of the $N$--th level is independent of the evolution $M$--th level. Equation \eqref{Gibbs} is a mathematical formulation of the statement that the dynamics preserves Gibbs measures.

\subsection{Restriction to centre}
Before continuing, we need to compute the states of certain
observables.

\begin{proposition}\label{Formula}
Let $\Pi$ denote the set of partitions of the set $\{1,\ldots,m\}$, let $\left| \pi\right|$
denote the number of blocks of the partition $\pi\in \Pi$ and let $B\in\pi$ mean that $B$ is 
a block in $\pi$. Then
$$
\langle E_{i_1j_1}\cdots E_{i_mj_m}\rangle_t = \sum_{\pi \in\Pi} 
t^{\left| \pi\right|}\prod_{\substack{ B\in\pi \\ B=\{b_1\ldots,b_k \}}}
1_{j_{b_1}=i_{b_2},j_{b_2}=i_{b_3},\ldots,j_{b_k}=i_{b_1}}
$$
\end{proposition}
\begin{example}
$$
\state{E_{21}E_{12}E_{21}E_{12}}_t = 2t^2+t
$$
with two contributing partitions having two blocks: $\{1,2\}\cup\{3,4\},\{1,4\}\cup\{2,3\}$, 
and one contributing partition having one block $\{1,2,3,4\}$.
\end{example}

\begin{example} 
$$
\state{E_{11}^3E_{22}}_t=t^4+3t^3+t^2
$$
with one contributing partition having four blocks:$\{1\}\cup\{2\}\cup\{3\}\cup\{4\},$
three contributing partitions having three blocks: $\{1,2\}\cup\{3\}\cup\{4\},\{1,3\}\cup\{2\}\cup\{4\},\{2,3\}\cup\{1\}\cup\{4\}, $
and one contributing partition having two blocks: $\{1,2,3\}\cup\{4\}.$
\end{example}

\begin{example} For any $m$,
$$
\state{E_{jj}^m}_t = B_m(t)
$$
where $B_m(t)$ is the $m$--th Bell polynomial. These are also the moments of a Poisson 
random variable with mean $t$, so under the state $\state{\cdot}_t$, each $E_{jj}$ can be
heuristically understood to be distributed as Poisson(t).
\end{example}

\begin{example} $$
\state{E_{11}E_{12}}_t=0
$$
with no contributing partitions.
\end{example}

\begin{proof}
Recall the defintion of $\state{\cdot}_t$ in \eqref{UsingTheFormula}. By Faa di Bruno formula, 
$$
\langle E_{i_1j_1}\cdots E_{i_mj_m}\rangle_t = \sum_{\pi\in \Pi} f^{(|\pi|)}(y)\prod_{B\in\pi}\frac{\partial^{|B|}y}{\prod_{b\in B}\partial x_b}\Bigg|_{x_1=\cdots=x_m=0}
$$
where
$$
f(y)=e^{ty}, \quad y = Tr(e^{x_1E_{i_1j_1}}\cdots e^{x_mE_{i_mj_m}}-Id).
$$
Note that
$$
f^{(|\pi|)}(y)\Bigg|_{x_1=\cdots=x_m=0} = t^{|\pi|}f(y)\Bigg|_{y=0}=t^{|\pi|}.
$$
Since we are taking the derivative with respect to $x_b$ and setting equal to $0$, 
we only need the linear terms in $x_b$, so it is equivalent to replace $y$ with 
$$
y = Tr((Id+x_1E_{i_1j_1})\cdots (Id+ x_mE_{i_mj_m})-Id).
$$
Here, $E_{ij}$ are the usual $N\times N$ matrices acting on $\C^N$, not the generators of $\U$.
Expanding the parantheses, all terms other than $Tr\left(\prod_{b\in B} x_bE_{i_bj_b}\right)$ do not contribute,
since these do not survive differentiation with respect to $x_b,b\in B$. Finally, since
$$
Tr\left(\prod_{\substack{ b\in B \\ B=\{b_1\ldots,b_k \}}}  E_{i_bj_b}\right)=
1_{j_{b_1}=i_{b_2},j_{b_2}=i_{b_3},\ldots,j_{b_k}=i_{b_1}},
$$
the proof is finished.
\end{proof}

In section 7 of  \cite{GKLLRT}, explicit generators of the centre $Z(\U)$ were found. See also chapter 7 of \cite{kn:M} for an exposition. 

Let $\mathcal{G}_m$ denote the directed graph with vertices and edges
$$
\{1,\ldots,m\} \quad \{(i,j):1\leq i,j\leq m\}. 
$$
Let $\Pi_k^{(m)}$ denote the set of all paths in $\mathcal{G}_m$ of length $k$ which start and 
end at the vertex $m$. For $\pi\in \Pi_k^{(m)}$ let $r(\pi)$ denote the length of the first
return to $m$. Let $E(\pi)\in U(\mathfrak{gl}_m)$ denote the element with coefficient $r(\pi)$ obtained by taking
the product when labeling
the edge $(i,j)$ with $E_{ij}$ when $i\neq j$, and the edge $(i,i)$ with $E_{ii}-m+1$.
For example, the path 
$$
\pi=\{5\rightarrow 3\rightarrow 3\rightarrow 1\rightarrow 5\rightarrow 5\rightarrow 2\rightarrow 5\}
$$
is in $\Pi^{(5)}_7$ with $r(\pi)=4$ and 
$$
E(\pi)=4E_{53}(E_{33}-4)E_{31}E_{15}(E_{55}-4)E_{52}E_{25}.
$$
Define the elements
$$
\Psi_k := \sum_{m=1}^N \sum_{\pi\in \Pi_k^{(m)}} E(\pi) \in U(\mathfrak{gl}_N).
$$
For example,
$$
\Psi_1 = \sum_{m=1}^N (E_{mm}-m+1), \quad \Psi_2 = \sum_{m=1}^N (E_{mm}-m+1)^2 + 2\sum_{1\leq l<m\leq N} E_{ml}E_{lm}.
$$
When we wish to emphasize that $\Psi_k\in U(\mathfrak{gl}_N)$, the notation $\Psi_{k}^{(N)}$ will be used.
\footnote{Caution: This notation is consistent with notation from integrable probability but different from
notation in representation theory.} 

\begin{theorem}\label{Gelfand} \cite{GKLLRT}
The centre $Z(\U)$ is generated by the elements $1,\{\Psi_k\}_{k\geq 1}$. Furthermore, the Harish--Chandra isomorphism maps $\Psi_k$ 
to the shifted symmetric polynomial $\sum_{m=1}^N (\la_m-m+1)^k$.
\end{theorem}
\begin{remark} Writing $\mathfrak{gl}_N = \mathfrak{n}_- \oplus \mathfrak{h} \oplus \mathfrak{n}_+$ 
where $\mathfrak{n}_+,\mathfrak{n}_-$ are the upper and lower nilpotent subalgebras and 
$\mathfrak{h}$ is the diagonal subalgebra, the Harish--Chandra homomorphism is the projection
$$
\U = (\mathfrak{n}_-\U + \U\mathfrak{n}_+)\oplus U(\mathfrak{h})\rightarrow U(\mathfrak{h})=S(\mathfrak{h})=\C[\la_1,\ldots,\la_N].
$$
This sends
$$
\Psi_k = \sum_{m=1}^N (E_{mm}-m+1)^k + (\text{other terms}) \mapsto \sum_{m=1}^N (E_{mm}-m+1)^k = \sum_{m=1}^N (\la_m-m+1)^k.
$$
Of course, $\sum_{m=1}^N (E_{mm}-m+1)^k$ is in general not central.
\end{remark}

Now it is time to explicitly state the relationship between the non--commutative random walk and the growing stepped surface. One may be tempted to think that 

$$
\begin{CD}
\U & @>P_t>> & \U\\
 @AAA &    & @AAA\\
U(\mathfrak{gl}_{N-1}) & @>P_t>> & U(\mathfrak{gl}_{N-1})
\end{CD}
$$

\vspace*{1\baselineskip}

\noindent is a non--commutative version of \eqref{CD}. However, care needs to be taken because the inclusion map
does not send $Z(U(\mathfrak{gl}_{N-1}))$ to $Z(\U)$.

A slight change of variables will make statements cleaner. If $p(\lambda)$ is a shifted symmetric 
polynomial, then by definition it is symmetric in the variables $x_i=\lambda_i-i+1$, and 
let $\bar{p}(x)$ denote the corresponding symmetric polynomial.

\begin{proposition} If $Y\in Z(\U)$ is sent to the symmetric polynomial $p_Y(x)$ by the Harish--Chandra
isomorphism, then
$$
\state{Y}_t = \E \left[\bar{p}_Y(X^{(N)}_1(t),\ldots,X^{(N)}_N(t))\right].
$$
\end{proposition}
\begin{proof}
This is not new, see \cite{BB}, but the proof is similar to 
Theorem \ref{TwoLevels} below, so will be repeated for clarity. By a result from \cite{BK},
$$
e^{t\Tr (U-\Id)} = \sum_{\lambda} \mathrm{Prob}(X_i^{(N)}(t) = \lambda_i-i+1,1\leq i\leq N)\frac{\chi_{\la}(U)}{\dim\la}
$$
where $\chi_{\la}$ and $\dim\la$ are the character and dimension of the highest weight representation $\la$. 
Thus, by linearity, 
\begin{align*}
\state{Y}_t &= \sum_{\la} \mathrm{Prob}(X_i^{(N)}(t) = \lambda_i-i+1,1\leq i\leq N) \frac{\state{Y}_{\chi_{\la}}}{\dim\la}\\
&= \sum_{\la} \mathrm{Prob}(X_i^{(N)}(t) = \lambda_i-i+1,1\leq i\leq N)  p_Y(\la_1,\ldots,\la_N)\\
&= \sum_x \mathrm{Prob}(X_i^{(N)}(t) = x_i,1\leq i\leq N)\bar{p}_Y(x_1,\ldots,x_N) 
\end{align*}
The last line is simply the right--hand--side of the proposition.
\end{proof}

\begin{proposition}\label{OneLevel} Suppose that  $P_t$ and $Q_t$ are two semigroups which preserve $Z(\U)$ and satisfy Theorem \ref{QuantumTheorem}(1),
then $P_tX=Q_tX$ for all $X\in Z(\U)$. In particular, $P_t$ is the Markov operator of the process 
$(X_1^{(N)}(t)>\ldots>X_N^{(N)}(t))$.
\end{proposition}
\begin{proof}
Theorem \ref{QuantumTheorem}(1) and (2) imply that $\state{P_tX}_s = \state{Q_tX}_s = \state{X}_{t+s}$ 
for all $s,t\geq 0$. In order to show $P_tX=Q_tX$, it suffices to show that if $Y\in Z$ 
satisfies $\state{Y}_t=0$ for all $t\geq 0$, then $Y=0$. Suppose this is not true, and let
$Y$ be a counterexample of minimal degree. But then $\state{P_tY}_s=\state{Y}_s=0$ and
Theorem \ref{Gelfand} implies that $ P_tY = Y + t(\text{lower degree terms})$. By assumption, 
$P_tY-Y\in Z$ also satisfies $\state{P_tY-Y}_s$ for all $s\geq 0$. Thus, since 
$Y$ is of minimal degree, $P_tY-Y=0$. 

If $Y$ has degree $d>1$, then by Theorem \ref{Gelfand} the term $E_{11}^d$ appears in $Y$. Thus
$P_tY$ has a $tdE_{11}^{d-1}$ term which cannot cancel with any term in $Y$. If $Y$ has degree
$1$ then $Y=a_1\Psi_1+a_0$ and $P_tY-Y = a_1tN$. Thus, there is a contradiction,
so no such $Y$ can exist. 

The second art of the proposition follows if we show that $Q_t$ preserves shifted symmetric polynomials.
But this follows because the process is the Doob $h$--transform of a random walk which is invariant under 
permuting the co--ordinates, and $h$ is anti--symmetric.
\end{proof}

\begin{example}
For $N=2$, one can explicitly compute (after a long calculation)
\begin{multline*}
P_t\Psi_4=\Psi_4 + 4t\Psi_3 + (6t^2+8t)\Psi_2 + 2t\Psi_1^2 \\
+ (4t^3+24t^2+10t)\Psi_1 + (2t^4+24t^3+38t^2+6t).
\end{multline*}
For instance, the only appearance of the monomial $E_{11}E_{22}$ in the right hand side is in $\Psi_1^2$. The only monomial in $\Psi_4$ that can lead to $E_{11}E_{22}$ is
$4E_{22}E_{21}E_{11}E_{12}$. The co--product $\Delta(E_{ij})=1\otimes E_{ij} + E_{ij}\otimes 1$ sends $E_{ij}$ either to the left tensor factor or the right tensor factor. In order to get $E_{11}E_{22}$, we must send $E_{22}E_{11}$ to the left and $E_{21}E_{12}$ to the right. Since $\state{E_{21}E_{12}}_t=t,$ the coefficient of $E_{22}E_{11}$ in $P_t\Psi_4$ must be $4t$. Since the coefficient of $E_{22}E_{11}$ in $\Psi_1^2$ is $2$, this implies that the coefficient of $\Psi_1^2$ in $\Psi_4$ is $2t$. Similar considerations can be applied to produce the other terms.

Now, evaluating this at $(4,2)$ with $t=3$ would predict that $$P_t\Psi_4(4,2)=257+12*65+78*17+6*5^2+354*5 + 1170=5453.$$
And indeed, the explicit determinantal formula from \cite{BF} for $N=2$ yields
$$
\frac{\sum_{b=x}^{\infty}\sum_{a=y}^b (b^k+(a-1)^k)(b-a+1)t^{b+a}\det\left[\begin{array}{cc}
(b-x)!^{-1} & (b-(y-1))!^{-1}  \\
(a-1-x)!^{-1} & (a-y)^{-1}  \end{array}\right]}{\sum_{b=x}^{\infty}\sum_{a=y}^b (b-a+1)t^{b+a}\det\left[\begin{array}{cc}
(b-x)!^{-1} & (b-(y-1))!^{-1}  \\
(a-1-x)!^{-1} & (a-y)^{-1}  \end{array}\right]}.
$$
A numerical computation at $(x,y)=(4,2),t=3,k=4$ with the sum from $b=x$ to $b=50\approx\infty$ yields
5452.999999999999999999999999999999418... Note that it is not obvious from the summation that the answer would even be an integer.
\end{example}

\begin{example}For general $N$, 
\begin{multline*}
P_t\Psi_1=\Psi_1+tN, \quad P_t\Psi_2=\Psi_2+2t\Psi_1+(t^2+Nt)N, \\
P_t\Psi_3=\Psi_3+3t\Psi_2+3(t^2+Nt)\Psi_1+N(t^3+3t^2N+\frac{1}{2}t(N^2+1)).
\end{multline*}
One can check explicitly that the semigroup property holds.
\end{example}

\begin{example}\label{P4} We wish to take asymptotics $N\approx \eta L$ and $t\approx \tau L$. We would get
\begin{align*}
P_t\Psi_{1,N} &= \Psi_1+\const, \\
P_t\Psi_{2,N} &=\Psi_2+2\tau L\Psi_1+\const, \\
P_t\Psi_{3,N} &= \Psi_3+3\tau L\Psi_2+3(\tau^2+\eta\tau)L^2\Psi_1+\const\\
P_t\Psi_{4,N} &= \Psi_4 + 4\tau L\Psi_3 + (6\tau^2+4\tau\eta)L^2\Psi_2 + (4\tau^3 + 12\tau^2\eta+2\tau\eta^2)L^3\Psi_1 \\
& \quad + 2\tau L \Psi_1^2 + \const\\
P_t\Psi_{1,N}^2 &= \Psi_{1,N}^2 + 2\eta\tau L^2\Psi_{1,N} + \const
\end{align*}
Again, one can check that the semigroup property holds.
\end{example}

\begin{theorem}\label{TwoLevels}
Suppose $Y_1\in Z(U(\mathfrak{gl}_{N_1})),\ldots, Y_r\in Z(U(\mathfrak{gl}_{N_r}))$ are mapped to the
symmetric polynomials $\bar{p}_{Y_1},\ldots\bar{p}_{Y_r}$ under the Harish--Chandra isomorphism. Assume
that $N_1\geq\ldots\geq N_r$ and $t_1\leq \ldots \leq t_r$. Then
$$
\state{Y_1P_{t_2-t_1}Y_2\cdots P_{t_r-t_1}Y_r}_{t_1}=\mathbb{E}\left[\bar{p}_{Y_1}(X^{(N_1)}(t_1))\cdots\bar{p}_{Y_r}(X^{(N_r)}(t_r))\right].
$$
\end{theorem}
\begin{proof}
In order to simplify notation and elucidate the idea of the proof, assume $r=2$. The more general case follows from exactly the same argument.

First prove it for $t_1=t_2$. Assume $N_1=N\geq M=N_2$. Let $m(\la,\mu)$ denote the multiplicity of $\mu$ in the restricted representation $V_{\la}\Big|_{U(M)}$. Use $\bar{m}(\cdot,\cdot)$ to denote the same quantity in the shifted co--ordinates $x_i=\la_i-i+1$. Then by the Gibbs property, that is \eqref{Gibbs},
\begin{align*}
\mathrm{RHS} 
&= \sum_{x^{(N)},x^{(M)}} \mathrm{Prob}(X^{(N)}(t)=x^{(N)}_i,X^{(M)}=x^{(M)}_j) \bar{p}_{Y_1}(x^{(N)})\bar{p}_{Y_2}(x^{(M)})\\
&= \sum_{x^{(N)},x^{(M)}} \mathrm{Prob}(X^{(N)}(t)=x^{(N)}_i)\frac{\bar{m}(x^{(N)},x^{(M)})h(x^{(M)})}{h(x^{(N)})} \bar{p}_{Y_1}(x^{(N)})\bar{p}_{Y_2}(x^{(M)})\\
\end{align*}
At the same time, 
$$
\state{Y_1Y_2}_t = \sum_{\la^{(N)}} \mathrm{Prob}(X^{(N)}(t)=\la^{(N)}_i-i+1) \frac{1}{\dim \la^{(N)}} \mathrm{Tr}(Y_1Y_2\Big|_{V_{\la^{(N)}}}).
$$
Since $Y_1$ is central, it acts as $p_{Y_1}(\la^{(N)})\Id$ on $V_{\la^{(N)}}$, so this equals
$$
\sum_{\la^{(N)}} \mathrm{Prob}(X^{(N)}(t)=\la^{(N)}_i-i+1) \frac{p_{Y_1}(\la^{(N)})}{\dim \la^{(N)}}  \mathrm{Tr}(Y_2\Big|_{V_{\la^{(N)}}}).
$$
By restricting $V_{\la^{(N)}}$ to $U(M)$ and using that $Y_2$ acts as $p_{Y_2}(\la^{(M)})\Id$ on $V_{\la^{(M)}}$, we get
$$
\sum_{\la^{(N)},\la^{(M)}} \mathrm{Prob}(X^{(N)}(t)=\la^{(N)}_i-i+1) \frac{m(\la^{(N)},\la^{(M)})\dim\la^{(M)}}{\dim \la^{(N)}} p_{Y_1}(\la^{(N)})  p_{Y_2}(\la^{(M)}).
$$
This is equal to the right--hand--side from above.

Now consider when $t=t_1\leq t_2=s$. Write $P_{s-t}Y_2$ as a sum over basis elements, that is $P_{s-t}Y_2=\sum_{\rho} c_{\rho}Y_{\rho}$. Then
\begin{align*}
\state{Y_1P_{s-t}Y_2}_t &= \sum_{\rho} c_{\rho}\state{Y_1Y_{\rho}}_t \\
&= \sum_{\rho}c_{\rho}\mathbb{E}\left[\bar{p}_{Y_1}(X^{(N)}(t))\bar{p}_{Y_{\rho}}(X^{(M)}(t))\right]\\
&= \E\left[\bar{p}_{Y_1}(X^{(N)}(t))(P_{s-t}\bar{p}_{Y_2})(X^{(M)}(t))  \right]
\end{align*}
Thus, it suffices to prove that
$$
\E\left[\bar{p}_{Y_1}(X^{(N)}(t))\bar{p}_{Y_2}(X^{(M)}(s))  \right]=\E\left[\bar{p}_{Y_1}(X^{(N)}(t))(P_{s-t}\bar{p}_{Y_2})(X^{(M)}(t))  \right]
$$
We have
\begin{align*}
&\E\left[\bar{p}_{Y_1}(X^{(N)}(t))\bar{p}_{Y_2}(X^{(M)}(s))  \right]\\
&=\sum_{y^{(N)},x^{(M)}} \bar{p}_{Y_1}(y^{(N)})\bar{p}_{Y_2}(x^{(M)}) \mathbb{P}(X^{(M)}(s) = x^{(M)},X^{(N)}(t) = y^{(N)})\\
&=\sum_{y^{(N)},y^{(M)},x^{(M)}} \bar{p}_{Y_1}(y^{(N)})\bar{p}_{Y_2}(x^{(M)}) \mathbb{P}(X^{(M)}(s) = x^{(M)},X^{(N)}(t) = y^{(N)},X^{(M)}(t) = y^{(M)})\\
&=\sum_{y^{(N)},y^{(M)},x^{(M)}} \bar{p}_{Y_1}(y^{(N)})\bar{p}_{Y_2}(x^{(M)}) \mathbb{P}(X^{(M)}(s) = x^{(M)}\vert X^{(N)}(t) = y^{(N)},X^{(M)}(t) = y^{(M)})\\
& \quad \quad \quad \quad \quad \times\mathbb{P}( X^{(N)}(t) = y^{(N)},X^{(M)}(t) = y^{(M)})
\end{align*}
By \eqref{Push} and the fact that $P_t=Q_t$, this then equals
\begin{align*}
&=\sum_{y^{(N)},y^{(M)},x^{(M)}} \bar{p}_{Y_1}(y^{(N)})\bar{p}_{Y_2}(x^{(M)}) \mathbb{P}(X^{(M)}(s) = x^{(M)}\vert X^{(M)}(t) = y^{(M)})\\
&\quad \quad \quad \quad \quad  \times \mathbb{P}( X^{(N)}(t) = y^{(N)},X^{(M)}(t) = y^{(M)}))\\
&=\sum_{y^{(N)},y^{(M)}} \bar{p}_{Y_1}(y^{(N)})(P_{s-t}\bar{p}_{Y_2})(y^{(M)}) \mathbb{P}( X^{(N)}(t) = y^{(N)},X^{(M)}(t) = y^{(M)}))\\
&=\E\left[\bar{p}_{Y_1}(X^{(N)}(t))(P_{s-t}\bar{p}_{Y_2})(X^{(M)}(t))  \right]
\end{align*}

\end{proof}

We wrap up this section by giving an example showing that although $P_t=Q_t$ on $Z(\U)$, they are not equal on subalgebras
generated by different $Z(\U)$. The determinantal formula from \cite{BF} yields
$$
Q_1(\Psi_1^{(2)}\Psi_1^{(1)})(\lambda^{(2)},\lambda^{(1)})\approx 2.37\ldots, \text{ when } \lambda^{(2)}=(1,0),\lambda^{(1)}=(0).
$$
However, $$P_{t}(\Psi_1^{(2)}\Psi_1^{(1)})=\Psi_1^{(2)}\Psi_1^{(1)}+2t\Psi_1^{(1)}+t\Psi_1^{(2)}+2t^2+t,$$
and when evaluated at $\lambda^{(2)}=(1,0),\lambda^{(1)}=(0),t=1$ yields $3$.

\section{Covariance Structure}
In this section, it will be shown that the central elements are asymptotically Gaussian with an explicit 
covariance that generalizes the Gaussian free field. Let us review some previously known results.

\begin{theorem}\label{SpaceLikePath}
\cite{BB,BF} Suppose $N_j=\lfloor \eta_j L\rfloor, t_j=\tau_j L$ for $1\leq j\leq r$. Assume they lie on a \textit{space--like path}, that is $N_1\geq\ldots\geq N_r$ and $t_1\leq\ldots \leq t_r$. Then as $L\rightarrow\infty$,
$$
\left(\frac{\Psi_{k_1}^{(N_1)}-\state{\Psi_{k_1}^{(N_1)}}_{t_1}}{L^{k_1}},\ldots,\frac{P_{t_r-t_1}\Psi_{k_r}^{(N_r)}-\state{P_{t_r-t_1}\Psi_{k_r}^{(N_r)}}_{t_1}}{L^{k_r}}\right) \rightarrow (\xi_1,\ldots,\xi_r),
$$
where the convergence is with respect to the state $\state{\cdot}_{t_1}$, and $(\xi_1,\ldots,\xi_r)$ is a Gaussian vector with covariance
$$
\displaystyle\E[\xi_i\xi_j]= \left(\frac{1}{2\pi i}\right)^2\iint_{\vert z\vert>\vert w\vert}\limits (\eta_i z^{-1} + \tau_i + \tau_i z)^{k_i} (\eta_j w^{-1} + \tau_j + \tau_j w)^{k_j} (z-w)^{-2}dzdw.
$$
\end{theorem}

The proof uses that the particle system is a determinantal point process along space--like paths. This
condition is necessary due to the construction using \eqref{CD}. In particular, there are no maps going up
from $S$ to $S^*$. A natural question is to ask what happens along time--like paths, that is, $N_1\leq N_2,t_1\leq t_2$. The main theorem
is
\begin{theorem}\label{TimeLikePath}
Suppose $N_j=\lfloor \eta_j L\rfloor, t_j=\tau_j L$ for $1\leq j\leq r$. Assume $\min(\ga_1,\ldots,\ga_r)=\ga_1$. Then as $L\rightarrow\infty$
$$
\left(\frac{\Psi_{k_1}^{(N_1)}-\state{\Psi_{k_1}^{(N_1)}}_{t_1}}{L^{k_1}},\ldots,\frac{P_{t_r-t_1}\Psi_{k_r}^{(N_r)}-\state{P_{t_r-t_1}\Psi_{k_1}^{(N_1)}}_{t_1}}{L^{k_r}}\right) \rightarrow (\xi_1,\ldots,\xi_r),
$$
where the convergence is with respect to the state $\state{\cdot}_{t_1}$, and $(\xi_1,\ldots,\xi_r)$ is a Gaussian vector with covariance
\[
\displaystyle\E[\xi_i\xi_j]= 
\begin{cases}
\displaystyle\left(\frac{1}{2\pi i}\right)^2\iint_{\vert z\vert>\vert w\vert}\limits (\eta_i z^{-1} + \tau_i + \tau_i z)^{k_i} (\eta_j w^{-1} + \tau_j + \tau_j w)^{k_j} (z-w)^{-2}dzdw,&\\
\hspace{3in} \eta_i\geq\eta_j,\ga_i\leq\ga_j&\\
\displaystyle\left(\frac{1}{2\pi i}\right)^2\iint_{\vert z\vert>\vert w\vert}\limits (\eta_j\frac{\tau_j}{\tau_i} z^{-1} + \tau_j + \tau_i z)^{k_j} (\eta_i w^{-1} + \tau_i + \tau_i w)^{k_i} (z-w)^{-2}dzdw,&\\
\hspace{3in} \eta_i<\eta_j,\ga_i\leq\ga_j&
\end{cases}
\]
\end{theorem}

\begin{example}\label{Ex}
The double integral can be computed using resides and the Taylor series 
$$(z-w)^{-2}=z^{-2}\left(1 + 2\frac{w}{z} + 3\frac{w^2}{z^2} + \ldots\right).$$
So for instance,
\begin{align*}
&\state{ \left(\frac{\Psi_{1}^{(\eta_1 L)}-\state{\Psi_{1}^{(\eta_1 L)}}_{\ga_1 L}}{L^{1}}\cdot\frac{P_{(\ga_2-\ga_1)L}\Psi_{1}^{(\eta_2 L)}-\state{P_{(\ga_2-\ga_1)L}\Psi_{1}^{(\eta_2 L)}}_{\ga_1 L}}{L^{1}}\right) }_{\ga_1 L}\\
&\rightarrow \ga_1\min(\eta_1,\eta_2).
\end{align*}
This can be checked using Proposition \ref{Formula}. Assume without loss of generality that $\eta:=\eta_1\leq \eta_2$ and set $\ga=\ga_1$. Since $\state{E_{ii}E_{jj}}= \state{E_{ii}}\state{E_{jj}}$ for $i\neq j$, then

\begin{align*}
&  \quad \lim_{L\rightarrow\infty}  L^{-2}\state{\left(\sum_{i=1}^{\lfloor \eta_1 L \rfloor } (E_{ii}-\ga_1 L)\right)\left(\sum_{j=1}^{\lfloor \eta_2 L\rfloor} (E_{jj} - \ga_1 L)\right)}_{\ga_1 L}\\
& =\lim_{L\rightarrow\infty}  L^{-2}\state{\left(\sum_{i=1}^{\lfloor \eta_1 L \rfloor } E_{ii}-\eta_1  \ga_1 L^2\right)\left(\sum_{j=1}^{\lfloor \eta_1 L\rfloor} E_{jj} - \eta_1 \ga_1 L^2\right)}_{\ga_1 L}\\
&= \lim_{L\rightarrow\infty}  \frac{\left( \eta L(\tau^2L^2 + \tau L)  +\eta L(\eta L - 1)(\tau L)^2 -2\eta\tau\cdot\eta\tau L^4 + \eta^2\tau^2 L^4 \right)}{L^2}\\
&= \eta\tau
\end{align*}
\end{example}

The remainder of this section will prove Theorem \ref{TimeLikePath}. By Theorem 5.1 of \cite{BB} and the
fact that $P_t$ preserves the centre, it is immediate that convergence to a Gaussian vector holds. It only 
remains to compute the covariance.

From the presence of $\Psi_1^2$ in $P_t\Psi_4$, it is necessary to understand products of $\Psi_k$. Heuristically, 
if $\Psi_1 \approx cL^2 + \xi L$, where $\xi$ is a Gaussian random variable, then
$\Psi_1^2 \approx c^2L^4+2c\xi L^3$. Here are two examples which demonstrate this:
\begin{example}
Since
$$
\lim_{L\rightarrow\infty}L^{-2}\state{\Psi_1^{(\eta L)}}_{\tau L} = (\tau \eta -\frac{1}{2} \eta^2),
$$
the heuristics would predict that
\begin{align*}
& \quad \quad \lim_{L\rightarrow\infty} \state{\frac{\left(\Psi_1^{(\eta L)}-\state{\Psi_1^{(\eta L)}}_{\tau L}\right)\left(\left[\Psi_1^{(\eta L)}\right]^2-\state{\left[\Psi_1^{(\eta L)}\right]^2}_{\tau L}\right)}{L^4}}_{\tau L} \\
&=\left(2\tau \eta - \eta^2\right)\lim_{L\rightarrow\infty} \state{\frac{\left(\Psi_1^{(\eta L)}-\state{\Psi_1^{(\eta L)}}_{\tau L}\right)\left(\Psi_1^{(\eta L)}-\state{\Psi_1^{(\eta L)}}_{\tau L}\right)}{L^2}}_{\tau L}\\
&= \left(2\tau \eta - \eta^2\right)\eta \tau 
\end{align*}
And indeed, an explicit calcuation yields
\begin{align*}
&\quad \lim_{L\rightarrow\infty} L^{-4}\state{\left(\sum_{i=1}^{\eta L} E_{ii}-\eta\tau L^2\right)\sum_{j,k=1}^{\eta L} E_{jj}E_{kk} + \left(\sum_{i=1}^{\eta L} E_{ii} - \eta\tau L^2 \right)(-\eta^2 L^2)\sum_{j=1}^{\eta L} E_{jj}}_{\tau L}\\
&=\lim_{L\rightarrow\infty} L^{-4}\Big((\eta \tau L^2 + 3L^4\eta^2\tau^2 + \eta^3\tau^3 L^6)-\eta\tau L^2(\eta^2\tau^2L^4 + \eta \tau L^2)\\
& \quad \quad - (\eta^2L^2)(\eta^2\tau^2 L^4 + \eta \tau L^2 - \eta\tau L^2\cdot\eta\tau L^2\Big)\\
&= (2\tau \eta-\eta^2)\eta\tau.
\end{align*}
\end{example}

\begin{example} Consider Theorem \ref{SpaceLikePath} with $r=2,k_1=3,k_2=4$. Using the formula for $P_t \Psi_4$ from Example \ref{P4} and 
replacing $\Psi_1^2$ with  $(2\tau_1 \eta_2-\eta_2^2)\Psi_1$ yields 
\begin{align*}
&\ \quad 12\eta_2\tau_1^2\tau_2(\eta_2^2\tau_1 + \tau_1(3\eta_2\tau_2 + 2\eta_2^2) + \tau_2(\tau_2+3\eta_2)(\tau_1+\eta_1))\\
&=12\eta_2\tau_1^3(\eta_2^2\tau_1 + \tau_1(3\eta_2\tau_1 + 2\eta_2^2) + \tau_1(\tau_1+3\eta_2)(\tau_1+\eta_1))\\
&+ 4(\tau_2-\tau_1)\cdot 3\eta_2\tau_1^2(\eta_2^2\tau_1+6\eta_2\tau_1^2+3\tau_1(\eta_2+\tau_1)(\eta_1+\tau_1))\\
&+(6(\tau_2-\tau_1)^2+4(\tau_2-\tau_1)\eta_2)\cdot 6\eta_2\tau_1^2(\tau_1(\tau_1+\eta_1)+\eta_2\tau_1)\\
&+(4(\tau_2-\tau_1)^3+12(\tau_2-\tau_1)^2\eta_2+2(\tau_2-\tau_1)\eta_2^2)\cdot 3\eta_2\tau_1^2(\tau_1+\eta_1)\\
&+2(\tau_2-\tau_1)\cdot (2\tau_1 \eta_2-\eta_2^2)\cdot 3\eta_2\tau_1^2(\tau_1+\eta_1),
\end{align*}
which can be checked computationally.
\end{example} 

Given a partition $\rho=(\rho_1,\ldots,\rho_l)$, let its \textit{weight} $\wt(\rho)$ denote 
$\left| \rho \right| + l(\rho)=\rho_1+\ldots+\rho_l+l$, and let
$\Psi_{\rho}=\prod_{i=1}^l \Psi_{\rho_i}$. In the asymptotic limit, we should be able 
to replace $\Psi_{\rho}$ with a linear combination of $\Psi_{\rho_i}$. In the examples above,
$\Psi_1^2$ was replaced with $(2\tau \eta - \eta^2)\Psi_1$.

\begin{proposition} Let $\eta,\tau>0$ be fixed.
(1) Set $N=\lfloor \eta L\rfloor $ and $t=\tau L$. Then $\state{\Psi_{\rho,N}}_{t} =\Theta( L^{\wt(\rho)})$.

(2) There exist constants $c'_{k,\rho}(\tau,\eta)$ such that 
$$
P_{\tau L}\Psi_{k,N} = \sum_{\rho} (c'_{k,\rho}(\tau,\eta)+o(1))L^{k+1-\wt(\rho)}\Psi_{\rho}.
$$
where the sum is over $\rho$ with weight $\wt(\rho)\leq k+1$.

(3) For any $\tau_1>\tau_0,$ there exist constants $c_{kj}(\tau_1,\tau_0,\eta)$ such that
\begin{multline*}
\lim_{L\rightarrow\infty} \state{\frac{\Psi_m - \state{\Psi_m}_{\tau_0 L}}{L^m} \cdot \frac{P_{(\tau_1-\tau_0) L}\Psi_k 
-\state{P_{(\tau_1-\tau_0) L}\Psi_k}_{(\tau_1-\tau_0) L}}{L^k}}_{\tau_0 L} \\
= \lim_{L\rightarrow\infty}  \sum_{j=1}^k c_{kj}(\tau_1,\tau_0,\eta)\state{\frac{\Psi_m - \state{\Psi_m}_{\tau_0 L}}{L^m} \cdot \frac{\Psi_j - \state{\Psi_j}_{\tau_0 L}}{L^j}}_{\tau_0 L}
\end{multline*}
\end{proposition}
\begin{proof}
(1) This can be proved from \cite{BB}, but this will be an alternative proof.

By definition,
$$
\Psi_{\rho} = \sum_{m_1=1}^{\rho_1}\cdots \sum_{m_l=1}^{\rho_l} \sum_{\pi_1 \in \Pi_{\rho_1}^{(m_1)}}\cdots \sum_{\pi_l \in \Pi_{\rho_l}^{(m_l)}} E(\pi_1)\cdots E(\pi_l).
$$
Consider the sum over $l$--tuples $(\pi_1,\ldots,\pi_l)$ such that the paths $\pi_1,\ldots,\pi_l$ 
cross over a total of exactly $\nu$ distinct vertices.
There are $\binom{N}{\nu}=\Theta(L^{\nu})$ such $l$--tuples, so it
remains to estimate $\state{E(\pi_1)\cdots E(\pi_l)}_{\tau L}$. Let $\pi$ be 
the union of the paths $\pi_1,\ldots,\pi_l$. Decompose $\pi$ into the union of
$s$ simple cycles. By Proposition \ref{Formula}, $\state{E(\pi_1)\cdots E(\pi_l)}_{\tau L}=\bigO{L^s}$.
Decomposing $\pi_j$ into $s_j$ simple cycles, it is clear that
$s=s_1+\ldots+s_l$. If $\pi_j$ covers exactly $\nu_j$ vertices, then
elementary graph theory gives $s_j=\rho_j-\nu_j+1$. Since
$\nu_1+\ldots+\nu_l\geq \nu$, thus
$$
\state{\Psi_{\rho,N}}_{t} = \bigO{L^{\nu}L^{s_1+\ldots+s_l}}=\bigO{L^{\rho_1+\ldots+\rho_l+l}}.
$$

To get a lower bound, just
observe that the constant term in $\Psi_{\rho,N}$ is $\Theta(L^{\wt(\rho)})$.

(2) By Theorem \ref{QuantumTheorem}(5), $P_{\tau L}\Psi_{k}$ can expressed as a linear combination of $\Psi_{\rho}$. Taking
$\state{\cdot}_{L}$ and using that $\state{P_{\tau L}X}_L = \state{X}_{(1+\tau)L}$, it follows from (1) that only $\wt(\rho)\leq k+1$ terms 
have nonzero coefficieints.

(3) First apply (2) to the left--hand--side. Then, by (1), 
\begin{multline*}
\Psi_{\rho_1}\ldots \Psi_{\rho_l} - \state{\Psi_{\rho_1}\ldots \Psi_{\rho_l}}_{\tau_0 L} \\
= \sum_{j=1}^l \state{\Psi_{\rho_1}}_{\tau_0 L} \ldots \widehat{\state{\Psi_{\rho_j}}_{\tau_0 L}} \ldots \state{\Psi_{\rho_l}}_{\tau_0 L}\left(\Psi_{\rho_j}-\state{\Psi_{\rho_j}}_{\tau_0 L}\right)
+ \text{smaller order terms}
\end{multline*}
\end{proof}

Given a Laurent polynomial $p(w)$, let $p(w)[w^r]$ denote the cofficient of $w^r$ in $p(w)$. Using the expansion $(z-w)^{-2}=z^{-2}(1+2(w/z)+3(w/z)^2+\ldots)$ in Theorem \ref{SpaceLikePath} and taking residues, one obtains
$$
\sum_{l=1}^k c_{kl}(\tau_2,\tau_1,\eta_2)(\eta_2 w^{-1} + \tau_1 + \tau_1 w)^{l}[w^r] = (\eta_2 w^{-1} + \tau_2 + \tau_2 w)^k[w^r], r\leq -1.
$$
For example, for $k=3$ and $r=-1$, and using the expansion of $P_t\Psi_3$, this says
\begin{equation}\label{SpaceLikeExample}
1\cdot (3\eta_2^2\tau_1 + 3\eta_2\tau_1^2) + 3(\tau_2-\tau_1)\cdot 2\eta_2\tau_1 + 3((\tau_2-\tau_1)^2 + (\tau_2-\tau_1)\eta_2)\cdot\eta_2 = 3\eta_2^2\tau_2 + 3\eta_2\tau_2^2.
\end{equation}
We need a formula for $r\geq 1$. Theorem \ref{TimeLikePath} follows from the proposition below. 

\begin{proposition}
For $r\geq 1,$
$$
\sum_{l=1}^k c_{kl}(\tau_2,\tau_1,\eta_1)(\eta_1 z^{-1} + \tau_1 + \tau_1 z)^l[z^r] = (\eta_1\frac{\tau_2}{\tau_1}z^{-1} + \tau_2 + \tau_1 z)^k [z^r].
$$
\end{proposition}
\begin{proof}
We start with an illustrative example. For $k=3$ and $r=1$ we would want to show
\begin{equation}\label{AlreadyKnow}
1\cdot (3\eta_1\tau_1^2 + 3\tau_1^3) + 3(\tau_2-\tau_1)\cdot 2\tau_1^2 + 3((\tau_2-\tau_1)^2 + (\tau_2-\tau_1)\eta_1)\cdot\tau_1 = 3\eta_1\tau_1\tau_2 + 3\tau_1\tau_2^2.
\end{equation}
This can be checked directly, but in general the coefficients $c_{kl}$ are difficult to work with. 
Instead, we would like to show that it follows directly from the covariance formula along space--like paths.
Indeed, this can be done just by multiplying \eqref{SpaceLikeExample} by $(\tau_1/\eta_1)^r$. (And recall that
$\eta_2<\eta_1$ in \eqref{SpaceLikeExample} while $\eta_1<\eta_2$ in \eqref{AlreadyKnow}).

Let $S_l^{(r)} = \{ (\epsilon_1,\ldots,\epsilon_l)\in \{-1,0,+1\}^l: \epsilon_1 + \ldots + \epsilon_l=r\}$ and define
$$
\chi(j) = 
\begin{cases}
\eta_1, j=-1\\
\tau_1, j=0\\
\tau_1, j=1
\end{cases}
\quad
\chi'_{\text{ti}}(j) = 
\begin{cases}
\eta_1\frac{\tau_2}{\tau_1}, j=-1\\
\tau_2, j=0\\
\tau_1, j=1
\end{cases}
\quad
\chi'_{\text{sp}}(j)=
\begin{cases}
\eta_1, j=-1\\
\tau_2, j=0\\
\tau_2, j=1
\end{cases}
$$
With this notation, what we want to show is that
\begin{equation}\label{WTS}
\sum_{l=1}^k c_{kl}(\tau_2,\tau_1,\eta_1)\sum_{\vec{\epsilon}\in S_l^{(r)}} \prod_{j=1}^l \chi(\epsilon_j)
= \sum_{\vec{\epsilon'} \in S_k^{(r)}} \prod_{j=1}^k \chi'_{\text{ti}}(\epsilon_j'), r\geq 1.
\end{equation}
From \eqref{AlreadyKnow},
$$
\sum_{l=1}^k c_{kl}(\tau_2,\tau_1,\eta_1)\sum_{\vec{\epsilon}\in S_l^{(-r)}} \prod_{j=1}^l \chi(\epsilon_j)
= \sum_{\vec{\epsilon'} \in S_k^{(-r)}} \prod_{j=1}^k \chi'_{\text{sp}}(\epsilon_j'), r\geq 1.
$$
By sending $\epsilon_j\mapsto -\epsilon_j$, this is equivalent to
$$
\sum_{l=1}^k c_{kl}(\tau_2,\tau_1,\eta_1)\sum_{\vec{\epsilon}\in S_l^{(r)}} \prod_{j=1}^l \chi(-\epsilon_j)
= \sum_{\vec{\epsilon'} \in S_k^{(r)}} \prod_{j=1}^k \chi'_{\text{sp}}(-\epsilon_j'), r\geq 1.
$$
And since for all $r$,
$$
\sum_{l=1}^k c_{kl}(\tau_2,\tau_1,\eta_1)\sum_{\vec{\epsilon}\in S_l^{(r)}} \prod_{j=1}^l \chi(-\epsilon_j)=
\left(\frac{\eta_1}{\tau_1} \right)^r \sum_{l=1}^k c_{kl}(\tau_2,\tau_1,\eta_1) \sum_{\vec{\epsilon}\in S_l^{(r)}} \prod_{j=1}^l \chi(\epsilon_j)
$$
it thus follows that the left--hand--side of \eqref{WTS} equals
$$
\left(\frac{\tau_1}{\eta_1} \right)^r\sum_{\vec{\epsilon'} \in S_k^{(r)}} \prod_{j=1}^k \chi'_{\text{sp}}(-\epsilon_j'), r\geq 1.
$$
So it suffices to show that
$$
\left(\frac{\tau_1}{\eta_1}\right)^r \prod_{j=1}^k  \chi'_{\text{sp}}(-\epsilon_j') = \prod_{j=1}^k \chi_{\text{ti}}'(\epsilon_j'), \quad \text{ for all } \vec{\epsilon}'\in S_k^{(r)}, r\geq 1.
$$
Since $r=\left| \{\epsilon_j=1\}\right| - \left| \{\epsilon_j=-1\} \right|,$ it follows that the left--hand--side is
$$
\left(\frac{\tau_1}{\eta_1}\right)^r \cdot \eta_1^{\left| \epsilon_j=1\right|} \tau_2^{\left| \epsilon_j=0\right|} \tau_2^{\left| \epsilon_j=-1\right|} = \tau_1^r \eta_1^{\left| \epsilon_j=-1\right|} \tau_2^{\left| \epsilon_j=0\right|} \tau_2^{\left| \epsilon_j=-1\right|}.
$$
And similarly, the right--hand--side is
$$
\tau_1^{\left| \epsilon_j=1\right|} \tau_2^{\left| \epsilon_j=0\right|} \left(\eta_1\frac{\tau_2}{\tau_1}\right)^{\left| \epsilon_j=-1\right|}
=\tau_1^r \eta_1^{\left| \epsilon_j=-1\right|} \tau_2^{\left| \epsilon_j=0\right|} \tau_2^{\left| \epsilon_j=-1\right|}.
$$
\end{proof}

The formula in Theorem \ref{TimeLikePath} appears to be different from the formula in \cite{Bo}. In particular, the covariance along space--like paths is different from the covariance along time--like paths.
However, after rescaling from Brownian Motion to Ornstein--Uhlenbeck, i.e. replacing $\ga_i,\ga_j$ with $e^{2\ga_i},e^{2\ga_j}$ and multiplying by $e^{-\ga_j k_j}e^{-\ga_i k_i}$, the formula becomes
\[
\displaystyle\E[\xi_i\xi_j]= 
\begin{cases}
\displaystyle-\frac{1}{\pi}\frac{e^{\ga_j}}{e^{\ga_i}}\iint_{\vert z\vert>\vert w\vert}\limits (\eta_i z^{-1} + e^{\tau_i} + z)^{k_i} (\eta_j w^{-1} + e^{\tau_j} + w)^{k_j} (\frac{e^{\ga_j}}{e^{\ga_i}}z-w)^{-2}dzdw,&\\
\hspace{3in} \eta_i\geq\eta_j,\ga_i\leq\ga_j&\\
\displaystyle-\frac{1}{\pi}\frac{e^{\ga_j}}{e^{\ga_i}}\iint_{\vert z\vert>\vert w\vert}\limits (\eta_jz^{-1} + e^{\tau_j} + z)^{k_j} (\eta_i w^{-1} + e^{\tau_i} + w)^{k_i} (\frac{e^{\ga_j}}{e^{\ga_i}}z-w)^{-2}dzdw,&\\
\hspace{3in} \eta_i<\eta_j,\ga_i\leq\ga_j&
\end{cases}
\]
In both expressions, the $z$--contour is larger and corresponds to the higher level ($\eta_i$ in the first
case and $\eta_j$ in the second). Hence, by switching the subscripts $i$ and $j$ in $\eta$, the formula is the same in 
both cases. It also matches the formula in \cite{Bo} with the expression $e^{\ga_j-\ga_i}$ playing the 
role of $c(t_p,t_q)$.




\bibliographystyle{alpha}

\end{document}